\documentclass[11pt,twoside, a4paper]{amsart}
\usepackage{amssymb}
\date{\today}

\date{\today}
\usepackage{geometry}
\def\a{{\mathfrak a}}

\def\1{{\bf 1}}

\def\deg{\text{deg}\,}

\def\w{\wedge}

\def\dbar{\bar\partial}

\def\C{{\mathbb C}}
\def\w{{\wedge}}
\def\P{{\mathbb P}}

\def\S{{\mathcal S}}

\def\Cu{{\mathcal C}}
\def\F{{\mathcal F}}

\def\reg{{\rm reg\,}}

\def\Hom{{\rm Hom\, }}
\def\codim{{\rm codim\,}}

\def\Im{{\rm Im\, }}

\def\Ker{{\rm Ker\,  }}

\def\E{{\mathcal E}}

\def\Ok{{\mathcal O}}

\def\Re{{\rm Re\,  }}

\def\U{{\mathcal U}}

\def\J{{\mathcal J}}
\def\nbh{neighborhood }

\def\be{\begin{equation}}
\def\ee{\end{equation}}

\def\p{{\mathfrak p}}

\def\ass{\text{Ass}}

\def\BEF{\text{\tiny{bef}}}
\def\blow{{\nu}}

\newtheorem{thm}{Theorem}[section]
\newtheorem{lma}[thm]{Lemma}

\newtheorem{prop}[thm]{Proposition}

\theoremstyle{definition}

\theoremstyle{remark}

\newtheorem{preremark}[thm]{Remark}
\newtheorem{preex}[thm]{Example}

\newenvironment{remark}{\begin{preremark}}{\qed\end{preremark}}
\newenvironment{ex}{\begin{preex}}{\qed\end{preex}}

\numberwithin{equation}{section}

\title[On the effective membership problem for polynomial ideals]
{On the effective membership problem for polynomial ideals}

\begin{document}

\date{\today}

\author{Mats Andersson \& Elizabeth Wulcan}

\address{Department of Mathematics\\Chalmers University of Technology and the University of
Gothenburg\\S-412 96 G\"OTEBORG\\SWEDEN}

\email{matsa@chalmers.se \& wulcan@chalmers.se}

%%\subjclass{32A26, 32A27, 32B15,  32C30}

\thanks{The first author was
  partially supported by the Swedish
  Research Council. The second author was partially supported by
the Swedish  Research Council and by the NSF}

\begin{abstract} 

We discuss the possibility of representing elements in polynomial  
ideals in $\C^N$ with optimal degree bounds. Classical theorems due to Macaulay
and Max Noether say that such a representation is possible 
under certain conditions on the variety of the associated
homogeneous ideal. We present some variants of these results,
as well as generalizations to subvarieties of $\C^N$. 

\end{abstract}

\maketitle

\section{Introduction}\label{max}

Let $V$ be an algebraic subvariety of $\C^N$ of pure dimension $n$  
and let $F_1,\ldots,F_m$ be
polynomials in $\C^N$.  
We are interested in finding solutions to
the polynomial division problem 
\begin{equation}\label{hummer}
F_1Q_1+\cdots +F_mQ_m=\Phi 
\end{equation}
on $V$ with degree estimates, 
provided $\Phi$ is in the ideal $(F_j)$ on $V$. By a result of Hermann, \cite{Her}, if $\deg F_j\leq d$, there are polynomials $Q_j$ such that
$\deg(F_jQ_j)\le \deg \Phi+C(d,N)$, where $C(d,N)$ is like
$2(2d)^{2^N-1}$ for large $d$ and 
thus doubly exponential.
It is shown  in \cite{MM} (see also \cite[Example~3.9]{BayMum}) that in general this estimate cannot be substantially improved.

If one imposes conditions on $V$ and $F_j$ one can, however, obtain
much sharper estimates. The following two results in
$\C^n$ are classical. 

\smallskip 
\noindent 
\emph{If $F_1,\ldots, F_m$ are polynomials in $\C^n$ of degrees
  $d_1\geq \ldots\geq d_m$ with no common
zeros even at infinity and $\Phi$ is any polynomial, then one can
solve \eqref{hummer} with $\deg(F_jQ_j)\leq \max (\deg\Phi, d_1+\ldots +d_{n+1}
-n)$.}
\smallskip 

\noindent 
\emph{If $F_1,\ldots, F_n$ are polynomials in $\C^n$
such that their  common zero set   is discrete and does not intersect
the hyperplane at infinity,
and $\Phi$ belongs to the ideal $(F_j)$, then one can find polynomials
$Q_j$ such that \eqref{hummer} holds and $\deg (F_jQ_j)\le\deg\Phi$.}

\smallskip 

The first theorem is due to Macaulay, \cite{Macaul}, and the second one
is Max~Noether's AF+BG theorem, \cite{Noe}, originally stated for
$n=2$. Noether's result is clearly optimal.

\smallskip

In this paper we present extensions of these results to the case of
more general varieties $V\subset\C^N$, and also generalizations in
which we relax the condition on (the zero set of) the $F_j$. 
It grew out of our paper \cite{semester}, in
which we extended to the singular setting a framework for solving
polynomial ideal membership problems with residue techniques
introduced in \cite{A3} and further developed in
\cite{AG,EW1, EW2}, see below. The proofs in this paper follow the same
setup. 
However, at least some of the results also admit algebraic proofs, see
Remark ~\ref{lock}.

\smallskip

Throughout we
will let $X$ denote the closure of $V$ in $\P^N$, and $\reg X$ the
\emph{regularity} of $X$, see Section ~\ref{lotta} for the
definition. 
For each $F_j$ we let $f_j$ denote the induced 
section of $\Ok(\deg F_j)|_X$.

We begin with an extension of Macaulay' theorem to singular varieties;
this can easily be proved by standard arguments, cf.\ Remark ~\ref{lock}. 

\begin{thm}\label{macsats}
Let $V$ be an algebraic subvariety of $\C^N$, with closure $X$ in $\P^N$,
and let $F_1,\ldots, F_m$ be polynomials in $\C^N$ of 
degrees $d_1\geq \ldots \geq d_m$. 
Assume that $f_j$ have no common zeros on $X$. Then for each
polynomial $\Phi$ in $\C^N$ 
there are polynomials $Q_j$ such that 
\eqref{hummer} holds and 
\begin{equation*}
\deg(F_jQ_j)\le  \max(\deg\Phi, d_1+\cdots +d_{n+1} -(n+1) + \reg X).
\end{equation*} 
\end{thm}

If $X$ is
smooth, then $\reg X \leq (n+1)(\deg X-1)+1$; 
this is Mumford's bound, see, e.g.,  \cite[Example 1.8.48]{Laz}. 
If $X$ is Cohen-Macaulay in $\P^N$ (and $N$ is minimal) then $\reg
X\leq \deg X-(N-n)$, see, \cite[Corollary 4.15]{Eis2}. 
In particular, if $V=\C^n$ so that $X=\P^n$, then $\reg X=1$; thus we
get back Macaulay's theorem. For a discussion of bounds on $\reg X$
for a general $X$, see, e.g., \cite[Section~3]{BayMum}. 
\smallskip

Let $Z^f$ denote the common zero set of $f_1,\ldots, f_m$ in $X$. 
Moreover, let $X_\infty:=X\setminus V$. 
For smooth varieties we have the
following version of Max Noether's theorem.

\begin{thm}\label{maxsats} 
Let $V$ be an algebraic subvariety of $\C^N$ of dimension $n$ such that its closure
$X$ in $\P^N$ is smooth,
and let $F_1,\ldots, F_m$ be polynomials in $\C^N$ of degrees $d_1\geq
\ldots \geq d_m$. 
Assume that $m\leq n$, that
\begin{equation}\label{krax}
\codim (Z^f\cap V)\ge m, 
\end{equation} 
and that $Z^f$ has no irreducible component contained in
$X_\infty$. 
If $\Phi$ is a polynomial that belongs to the ideal $(F_j)$ in $V$,
then there is a representation \eqref{hummer} with 
\begin{equation}\label{allman}
\deg(F_jQ_j)\le  \max(\deg\Phi, d_1+\cdots + d_m-m + \reg X). 
\end{equation} 

If in addition $X$ is Cohen-Macaulay in $\P^N$ one can choose
$Q_j$ so that 
\begin{equation}\label{andra} 
\deg(F_jQ_j)\le  \deg\Phi. 
\end{equation} 
\end{thm}

\begin{remark}\label{skippa} 
If $X$ is Cohen-Macaulay it suffices that $V$ is smooth to
obtain \eqref{andra}. 
\end{remark}

For $V=\C^n$ Theorem \ref{maxsats} appeared in
\cite[Theorem~1.2]{A3}.

\smallskip 

For a general $X$, in order to have a Max Noether theorem, we need
the common zero set of the $f_j$ not to intersect the singular locus
of $X$ too badly. 
To make this statement more precise we need to introduce what we call
the 
\emph{intrinsic BEF-varieties}  
$$
X^{n-1}\subset\cdots \subset  X^1,
$$
of $X\subset \P^N$. These are the sets 
where the mappings in a locally free resolution of $\Ok^{\P^N}/\J_X$
do not have optimal rank. They are intrinsically defined subvarieties of
$X$ that are contained in
$X^0:=X_\text{sing}$. The codimension of $X^\ell$ is at least
$\ell+1$, and if $X$ is locally Cohen-Macaulay $X^\ell$ is
empty for $\ell\geq 1$, see Sections ~\ref{rescurr} and ~\ref{struktur}.

\begin{thm}\label{max2} 
Let $V$ be an algebraic subvariety of $\C^N$ of dimension $n$, with closure $X$ in
$\P^N$, and let $F_j$ be as in Theorem ~\ref{maxsats}. 
Assume that $Z^f$ satisfies \eqref{krax}, that $Z^f$ has no irreducible component contained in
$X_\infty$, and moreover that 
\begin{equation}\label{kraxa}
\codim(Z^f\cap X^\ell)\geq m+ \ell +1, ~~~~~ \quad \ell \geq 0. 
\end{equation}
If $\Phi$ is a polynomial that belongs to the ideal $(F_j)$ in $V$,
then there is a representation \eqref{hummer} such that  
\eqref{allman} holds. 
If in addition $X$ is Cohen-Macaulay in $\P^N$, and $m\le n$, we can
choose $Q_j$ such that \eqref{andra} holds. 
\end{thm}
Notice that \eqref{kraxa} forces that either $Z^f\cap X_{\text{sing}}=\emptyset$ or
$m<n$. If $X$ is smooth, then \eqref{kraxa} is vacuous, and thus
Theorem ~\ref{maxsats} follows immediately from Theorem ~\ref{max2}. 
If only $V$ is smooth but $X$ is Cohen-Macaulay, then by the assumption
on $Z^f$ $\codim(Z^f\cap X_\infty)\geq m+1$ and since $X^0\subset
X_\infty$, \eqref{kraxa} is satisfied. This proves the claim in Remark
~\ref{skippa}.

\smallskip

Next we will present some generalizations of Theorem ~\ref{max2} where
we relax the hypotheses on the common zero set $Z^f$ of the $f_j$. 
First, we drop the size hypothesis \eqref{krax} on $Z^f\cap V$. We
then 
still get an estimate of the form \eqref{allman} but the
second entry on the right hand side is now replaced by a constant that depends on $F_j$ in a
more involved manner. 
The condition that $Z^f$ has no irreducible
component 
at infinity should now be understood as that the ideal sheaf  $\J_f$ over $X$
generated by  the sections $f_1,\ldots, f_m$ has no associated
variety, in the sense of \cite{S}, contained in $X_\infty$, see
Section ~\ref{ass}. This means
that at each $x\in X_\infty$, $(\J_{f})_x$ has no (varieties of)
associated prime ideals contained in $X_\infty$. 
Let $J_f$ be the homogeneous ideal in $\C[z_0,\ldots, z_N]$ associated
with $\J_f$, and let $\reg J_f$ be the \emph{regularity} of $J_f$,
cf.\ Section ~\ref{lotta}.

\begin{thm}\label{noethersats}
Let $V$ be an algebraic subvariety of $\C^N$, with closure $X$ in
$\P^N$, and let $F_1,\ldots,F_m$ be polynomials in $\C^N$. 
Assume that $\J_f$ has no associated variety contained in $X_\infty$. 
Then there is a constant $\beta=\beta(X,F_1,\ldots,F_m)$ such that
if $\Phi\in (F_j)$,
then there is a representation \eqref{hummer} on $V$ with
\begin{equation}\label{singmax}
\deg(F_jQ_j)\le \max (\deg\Phi,   \beta).
\end{equation}
If $V=\C^N$, one can take $\beta= \reg J_f$. 

Conversely, if there is an associated prime of $\J_f$ contained in
$X_\infty$, then there is no $\beta$ such that one can solve
\eqref{hummer} with \eqref{singmax} for all $\Phi$ in $(F_j)$. 
\end{thm}

In ~\cite{Sh} Shiffman computed the regularity of a zero-dimensional
homogeneous polynomial ideal $J_f$ to be $\leq d_1+\ldots + d_{n+1}-n$. Using this he obtained (the first part
of) Theorem ~\ref{noethersats} for $V=\C^N$ and $\dim Z^f=0$ with 
$\beta=\reg J_f=d_1+\cdots +d_{n+1} -n$, i.e., the same bound as in
Macaulay's theorem, see \cite[Theorem~2(iv)]{Sh}. 
Theorem ~\ref{noethersats} can thus 
be seen as a generalization of Shiffman's result.

The estimate \eqref{singmax} is clearly sharp if $\deg\Phi\ge\beta$.
If  the ideal sheaf $\J_f$ is locally Cohen-Macaulay, for instance
locally a complete intersection, 
then there are no embedded primes of $\J_f$, and so the hypothesis
that $\J_f$ has no associated variety at infinity just means that no irreducible component of
$Z^f$ is contained in $X_\infty$. Thus we get
back the hypothesis in Theorems ~\ref{maxsats} and ~\ref{max2}.

\smallskip 

Next, let us instead relax the condition that $Z^f$ has no  
irreducible components at infinity. 
If the degrees of $F_j$ are $\leq d$, we let
$\tilde f_j$ denote the section of $\Ok(d)|_X$ corresponding to $F_j$.
We let $Z^{\tilde f}$ be the common zero set of $\tilde f_1, \ldots,
\tilde f_m$ and $\J_{\tilde f}$ the coherent analytic sheaf over $X$ generated by 
the $\tilde f_j$. Moreover, we  
let $c_\infty$ be the maximal codimension of  the
so-called  {\it (Fulton-MacPherson) distinguished varieties} of
$\J_{\tilde f}$ 
that are contained in 
$X_\infty$, 
see Section~\ref{dist}.
If there  are no distinguished varieties contained in  $X_\infty$,
then we interpret $c_\infty$ as $-\infty$. Note that it is
not sufficient that $Z^{\tilde f}\cap V=Z^{\tilde f}$, since there may be embedded
distinguished varieties contained in $X_\infty$.  
It is well-known that the codimension of a distinguished variety
cannot exceed the number $m$, see, e.g., Proposition~2.6 in \cite{EL}, and thus
$c_\infty\le\mu$, where 
\begin{equation*}
\mu:=\min (m, n).
\end{equation*}

\begin{thm}\label{halvglattnoether}
Let $V$ be an algebraic subvariety of $\C^N$, 
with closure $X$ in $\P^N$, and let $F_1,\ldots, F_m$
be polynomials in $\C^N$ of degree $\le d$. 
Assume that $Z^{\tilde f}$ satisfies 
\begin{equation}\label{krax2}
\codim (Z^{\tilde f}\cap X)\geq m
\end{equation}
 and 
\begin{equation}\label{paxa}
\codim(Z^{\tilde f}\cap X^\ell)\geq m+ \ell +1, ~~~~~ \quad \ell \geq 0. 
\end{equation} 
If $\Phi$ is a polynomial that belongs to $(F_j)$ on $V$, then there
is a representation \eqref{hummer} on $V$ with
\begin{equation}\label{kyckling} 
\deg(F_jQ_j)\le\max(\deg\Phi+\mu d^{c_\infty}\deg X, (d-1)\min(m,n+1)+\reg X).
\end{equation}

If in addition $X$ is locally Cohen-Macaulay in $\P^N$ and $m\le n$, then we
can choose $Q_j$ such that 
\begin{equation*}
\deg(F_jQ_j)\le \deg\Phi+m d^{c_\infty}\deg X.
\end{equation*} 
\end{thm}

Note that for most choices of $F_j$ and $\Phi$ the first entry in
\eqref{kyckling} is much larger than the second entry. 
For instance this is true for all $\Phi$ if $c_\infty \geq 2$ and $d$ is large
enough. In particular, if
$X=\P^n$, so that $\reg X=1$, and $c_\infty \geq 2$, the first entry 
is the largest for all $d$.

For $X=\P^n$ Theorem ~\ref{halvglattnoether} is due to the first author
and G\"otmark, \cite[Theorem~1.3]{AG}. 
In the case when $\deg F_j=d$, so that $\tilde f_j=f_j$, Theorem
~\ref{halvglattnoether} generalizes Theorems ~\ref{macsats}
--~\ref{max2}, see Remark ~\ref{gron}.

\begin{ex}\label{jelo} 
If the $F_j$ have no common zeros on $V$, then
Theorem~\ref{halvglattnoether} gives a solution to 
\begin{equation*}
F_1Q_1+\cdots + F_mQ_m=1
\end{equation*}
with 
$\deg (F_j Q_j)\leq \mu d^{\mu}\deg X$ if $d$ is large
enough. 
Except for the annoying factor $\mu$ we then get back 
is Jelonek's optimal effective Nullstellensatz, \cite{Jel}. 
\end{ex}

\smallskip

\smallskip 

Note that the estimates of $\deg (F_jQ_j)$ in the theorems above hold for
representations of \emph{all} $\Phi$ in $(F_j)$. 
If one, instead of adding conditions on $V$ and $F_j$, imposes further
conditions on $\Phi$, then Hermann's degree estimate
for solutions to \eqref{hummer} can also be essentially improved.  
Theorem ~1.1 in our recent paper \cite{semester} asserts that for any
$V\subset \C^N$ there is a number $\mu_0$ such that if $F_1,\ldots, F_m$ are
polynomials in $\C^N$ of degree $\leq d$ and $\Phi$ is a polynomial such that 
$|\Phi|\leq C |F|^{\mu + \mu_0}$ locally on $V$, 
where $|F|^2=|F_1|^2+\cdots + |F_m|^2$, then one can solve
\eqref{hummer} with 
\begin{equation}\label{nunu}
\deg (F_j Q_j)\leq 
\max \big(\deg \Phi + (\mu+\mu_0) d^{c_\infty} \deg X, (d-1)\min(m,n+1) +\reg X\big).
\end{equation}
The statement that $|\Phi|\leq C
|F|^{\mu + \mu_0}$ implies that there is a representation
\eqref{hummer} is a direct consequence of Huneke's uniform Brian\c
con-Skoda theorem, \cite{BS, Hun}, and thus the degree estimate
\eqref{nunu} can be seen as
a global effective Brian\c con-Skoda-Huneke theorem.

\medskip 
\noindent 
\textbf{Acknowledgment.}
We thank Richard L\"ark\"ang for helpful discussions.

\section{Residue currents}\label{prelim}

We will briefly recall some residue theory. 
For more details we refer to \cite{semester} and the references
therein.

\subsection{Currents on a singular variety}\label{currents} 

If nothing else is mentioned $X$ will be a reduced subvariety of
$\P^N$ of pure dimension $n$.
The sheaf
$\Cu_{\ell,k}$ of  currents of bidegree $(\ell,k)$ on $X$ is by definition the dual
of the sheaf  $\E_{n-\ell,n-k}$ of smooth  $(n-\ell,n-k)$-forms on $X$.
If $i\colon X\to \P^N$ is an embedding of $X$, then
$\E_{n-\ell,n-k}$ can be identified with the quotient sheaf
$\E_{n-\ell,n-k}^{\P^N}/\Ker i^*$, where $\Ker i^*$ is the sheaf of forms $\xi$ on $\P^N$ such that
$i^*\xi$ vanish on  $X_{\text{reg}}$. It follows that the currents $\tau$ in
$\Cu_{\ell,k}$ can be identified with
currents $\tau'=i_*\tau$ on $\P^N$ of bidegree $(N-n+\ell,N-n+k)$  that
vanish on $\Ker i^*$.

Given a holomorphic function $f$ on $X$, we write 
 $1/f$ for the \emph{principal value
  distribution}, defined for instance as  
$\lim_{\epsilon\to 0}\chi(|f|^2/\epsilon)(1/f),$ 
where $\chi(t)$ is the characteristic function of the
interval $[1,\infty)$ or a smooth approximand of it, or as the
  analytic continuation of $\lambda\to|f|^{2\lambda}(1/f)$ to $\lambda=0$. 
It is readily checked that $f(1/f)=1$ as distributions and that the
\emph{residue current} $\dbar(1/f)$ satisfies $f \dbar(1/f)=0$.
We will need the fact that 
\begin{equation}\label{nkel}
v^\lambda|f|^{2\lambda}\left.\frac{1}{f}\right |_{\lambda=0} = \frac{1}{f}
\end{equation}
if $v$ is a strictly positive smooth function; cf.\ \cite[Lemma~2.1]{A2}.

\subsection{Pseudomeromorphic currents}\label{pseudo} 

The notion of pseudomeromorphic currents on manifolds was introduced in \cite{AW2}. A slightly extended version appeared in
\cite{AS}:  
A current on $X$ is \emph{pseudomeromorphic} if it is (the sum of terms
that are) the
pushforward under (a composition of) modifications, projections, and 
open inclusions of 
currents of the form 
$$
\frac{\xi}{s_1^{\alpha_1}\cdots s_{n-1}^{\alpha_{n-1}}}\w\dbar\frac{1}{s_n^{\alpha_n}},
$$
where $s$ is a local coordinate system and $\xi$ is a smooth form with
compact support, see, e.g., \cite{AS} for details.

Pseudomeromorphic currents in many respects behave like positive
closed currents. For example they satisfy the 
\noindent {\it dimension principle:  If $\tau$ is a pseudomeromorphic current
on $X$ of bidegree $(*,p)$ that has support on a variety of codimension
$>p$, then $\tau=0$.}

Also, pseudomeromorphic currents allow for multiplication with
characteristic functions of constructible sets so that ordinary
computational rules hold. 
If $\tau$ is a pseudomeromorphic current on $X$
and $V$ is a subvariety of $X$, 
then the natural restriction
of $\tau$ to the open set $X\setminus V$ has a canonical extension
$\mathbf 1_{X\setminus V}\tau:= |h|^{2\lambda}\tau|_{\lambda=0}$, where $h$ is any
holomorphic tuple such that $\{h=0\}=V$. It follows that $\mathbf 1_{V}\tau:=
\tau-\mathbf 1_{X\setminus V}\tau$ is a pseudomeromorphic current with support on $V$. 
Note that if $\alpha$ is a smooth form, then
${\bf 1}_{V}\alpha\w\tau=\alpha\w{\bf 1}_{V}\tau$ and if $W$ are $W'$
are constructible sets, then 
\begin{equation}\label{dubbelstjarna}
{\bf 1}_W{\bf 1}_{W'}\tau={\bf 1}_{W\cap W'}\tau.
\end{equation}
Moreover, if $\pi\colon \widetilde X\to X$ is a modification,  $\tilde\tau$ is a
pseudomeromorphic current on $\widetilde X$,
and
$\tau=\pi_*\tilde\tau$, then
\begin{equation}\label{enkelstjarna}
{\bf 1}_V\tau=\pi_*\big({\bf 1}_{\pi^{-1}V}\tilde\tau\big)
\end{equation}
for any subvariety $V\subset X$. 
If $W$ is a subvariety of $X$ and ${\bf 1}_V \tau=0$ for all subvarieties $V\subset W$ of positive
codimension we say that $\tau$ has the 
{\it the standard extension
property}, SEP with respect to $W$, see \cite{Bj}.

Recall that a current is {\it semi-meromorphic}  if it is the quotient of a smooth
form and a holomorphic function. 
Following \cite{AS} we say that a current $\tau$ is {\it almost semi-meromorphic}
in $X$ if there is a modification
$\pi\colon \widetilde X\to X$ and a  semi-meromorphic current
$\tilde \tau$  such that  $\tau=\pi_*\tilde\tau$.

\subsection{Residue currents associated with Hermitian complexes}\label{rescurr} 

Consider a complex of Hermitian holomorphic vector bundles over a
complex manifold ~$Y$ of dimension $n$, 
\begin{equation}\label{ecomplex}
0\to E_M\stackrel{f^M}{\longrightarrow}\ldots\stackrel{f^3}{\longrightarrow}
E_2\stackrel{f^2}{\longrightarrow}
E_1\stackrel{f^1}{\longrightarrow}E_0\to 0, 
\end{equation}
that is pointwise exact outside an analytic variety ~$Z\subset Y$ of positive
codimension $p$. Suppose that the rank of ~$E_{0}$ is $1$. 
In \cite{A, AW1} was associated to \eqref{ecomplex} a $\bigoplus \Hom(E_0,E_k)$-valued pseudomeromorphic current $R=R^f$; it has support on $Z$ and in a certain sense it measures the
lack of exactness of the associated sheaf complex of holomorphic sections 
\begin{equation}\label{sheaves}
0\to \mathcal{O}(E_M)\stackrel{f^M}{\longrightarrow}\ldots\stackrel{f^3}{\longrightarrow}
\mathcal{O}(E_2)\stackrel{f^2}{\longrightarrow}
\mathcal{O}(E_1)\stackrel{f^1}{\longrightarrow}\mathcal{O}(E_0).
\end{equation}

\begin{prop}\label{glatta}
If $\phi$ is a holomorphic section of $E_0$ such that $R\phi=0$, then
$\phi\in\Im f^1$. Moreover, if 
\begin{equation}\label{cohovillkor}
H^{k-1}(Y,\Ok(E_k))=0, \quad  1\le k\le\min(M,n+1), 
\end{equation}
then there is a global holomorphic section $q$ of $E_1$ such that  $f^1q=\phi$.  
\end{prop}

We also have the \emph{duality principle: If \eqref{sheaves} is exact,
i.e., if it is a locally free resolution of the sheaf $\Ok(E_{0})/\Im
f^1$, then $R\phi =0$ if and only if $\phi\in\Im
f^1$.}

As in \cite{semester} we will refer to a (locally) free resolution
\eqref{sheaves} of
$\Ok(E_0)/\J$ together with Hermitian metrics on the corresponding
vector bundles as a \emph{Hermitian (locally) free resolution}.

\smallskip

Let us look at the construction of $R$ in a special case; see, e.g.,
\cite{semester} for more details and the general case. Let ~$R_k$ denote the component of ~$R$ that takes values in
$\Hom(E_0, E_{k})$.  

\begin{ex}[The Koszul complex]\label{koszulex}

Given Hermitian line bundles $S\to Y$ and $L_1,\ldots, L_m\to Y$ and a tuple $f$ of holomorphic sections $f_1,\ldots,
f_m$ of $L_1,\ldots, L_m$, respectively, let \eqref{ecomplex} be the
(twisted) Koszul
complex of $f$: 
Let $E^j$ be disjoint trivial line bundles with basis
elements $e_j$, let $E=L^{-1}_1\otimes E^1 \oplus\cdots \oplus
L_m^{-1}\otimes E^m$, and 
identify $f$ with a section $f=\sum f_j e^*_j$ of $E^*$, where
$e_j^*$ are the dual basis elements. 
Moreover, let 
$$
E_0=S, \quad     E_k=S\otimes \Lambda^k E, 
$$
and let all $f^k$ in \eqref{ecomplex} be interior multiplication
$\delta_f$ by the section $f$.

The current associated with the Koszul complex was introduced in
\cite{A2}; we will briefly recall the construction. 
Let $\sigma$ be  the section of $E$ over $Y\setminus Z$
with pointwise minimal norm such that $f\cdot\sigma=\delta_f\sigma=1$, i.e.,
$$
\sigma = \sum_j \frac {f_j^* e_j}{|f|^2},
$$
where $f_j^*$ is  the section of $L_j^{-1}$ of minimal norm such that
$f_jf_j^*=|f_j|^2_{L_j}$, and $|f|^2=|f_1|_{L_1}^2+\cdots +|f_m|_{L_m}^2$. 
Then $R_{k}$ equals the analytic continuation to $\lambda=0$ of 
\begin{equation}\label{mung}
R^\lambda=R^{f,\lambda}:=\dbar|f|^{2\lambda} \w \sigma\w (\dbar \sigma)^{k-1}.
\end{equation} 
Here the exterior product is with respect to the exterior algebra over $E\oplus T^*(Y)$ so that
$d \bar z_j\w e_\ell =-e_\ell\w d\bar z_j$ etc; in particular,  
$\dbar\sigma$ is a form of even degree.

If $m=1$, then $\sigma$ is just $(1/f_1) e_1$
and $R=\dbar(1/f_1) \w e_1$. 
In general, the coefficients of $R$ are the Bochner-Martinelli
residue currents introduced by Passare-Tsikh-Yger \cite{PTY}. 
The sheaf complex  associated with the Koszul
complex is exact
if and only if $f$ is a \emph{complete
  intersection}, i.e.,  $\codim
Z^f=m$. In this case one can prove that (the coefficient of) $R=R_m$ coincides with the classical \emph{Coleff-Herrera residue current} 
$\dbar(1/f_1)\w\cdots\w\dbar
(1/f_m)$.

\end{ex}

Since, in light of the above example, $R$ generalizes the classical
Coleff-Herrera residue current (as well as the Bochner-Martinelli
residue currents), we say that $R$ is the \emph{residue current}
associated with the Hermitian complex \eqref{ecomplex}.

The construction of $R$ in general involves the minimal inverse
$\sigma_k$ of each $f^k$ in \eqref{ecomplex}; $R$ is defined as the
analytic continuation to $\lambda=0$ of a 
regularization $R^\lambda$ which generalizes \eqref{mung}. The component $R_k$
is of the form $\dbar|f|^{2\lambda}\w\sigma_k \dbar
\sigma_{k-1}\cdots\dbar \sigma_{1}|_{\lambda=0}$; see, e.g., \cite{AW1} for a 
precise interpretation of this. 
It follows that outside the set $Z_{k}$
where $f^{k}$ does not have optimal rank, 
\begin{equation}\label{plommon}
R_{k}=\alpha_{k} R_{k-1}, 
\end{equation} 
where $\alpha_{k}$ is a smooth $\Hom(E_{k-1},
E_{k})$-valued $(0,1)$-form. 
If \eqref{sheaves} is exact,
these sets are independent of the resolution; we call them  {\it BEF
  varieties} (which is an acronym for Buchsbaum-Eisenbud-Fitting,
cf.\ \cite{semester}) and denote them $Z_k^{\BEF}=Z_k^\BEF(\J_f)$. The Buchsbaum-Eisenbud theorem asserts  that
$\codim Z_k^\BEF\ge k$; more precisely it says that the complex
\eqref{sheaves} is exact if and only if the codimension of the set
where $f_k$ does not have optimal rank is $\geq k$, see, e.g.,
\cite[Theorem~3.3]{Eis2}. 
If $\J_f$ has pure codimension $p$, then
$\codim Z_k^\BEF\geq k+1$ for $k> p$, see
\cite[Corollary~20.14]{Eis}. Also, note that if in addition $X$ is
locally Cohen-Macaulay, then  $Z_k=\emptyset$ for $k>p$. 
The current $R_k$ has bidegree $(0,k)$, and
thus, by the dimension
principle, $R_k=0$ for $k<p$, and for degree reasons, $R_k = 0$ for $k>n$.

If the complex \eqref{ecomplex} is twisted by a Hermitian
line bundle, the residue current $R$ is not affected. This follows
since the $\sigma_k$ are not affected by the
twisting.

\subsection{BEF-varieties on singular varieties}\label{singbef}

Let $i: X\to Y$ be a (local) embedding of $X$ of dimension $n$ into a
smooth manifold $Y$ of dimension $N$. 
Note that if $\J_f$ is a coherent ideal sheaf on $X$, then
$\J_f+\J_X$ is a well-defined sheaf on $Y$. Indeed, locally $\J_f$ is
the pullback $i^*\widetilde \J_f$ of an ideal sheaf on $Y$ and the
sheaf $\widetilde \J_f+\J_X$ is independent of the choice of 
$\widetilde \J_f$. 
We define \emph{$k$th BEF-variety} $Z_k^\BEF(\J_f)$ of $\J_f$ as 
$Z_{k+N-n}^\BEF({\J_f + \J_X})$, which clearly is a subvariety of $X$.

This definition is independent of the embedding
$i$. To see this recall that (locally) $i$ can be factorized as 
$X\stackrel{\iota}{\to}\Omega\stackrel{}{\to}\Omega\times \C^r=Y$,
where $\iota$ is a minimal embedding.  
From a locally free
resolution of $\Ok^\Omega/\J$, where $\J$ is a coherent ideal sheaf
over $\Omega$, it is not hard to construct a locally free resolution
of $\Ok^Y/(\J+\J_\Omega)$. By relating the sets where the mappings in these
resolutions do not have have optimal rank one can show that the BEF-varieties
of $\J$ are independent of $i$, cf. \cite[Remark~4.6]{A12} and
\cite[Section~3]{semester}.

\subsection{The structure form $\omega$ on a singular
  variety}\label{struktur}

Now assume that $X$ is as in Section ~\ref{currents}, and let $R$ be the
residue current associated with a Hermitian free resolution $\Ok(E_\bullet),
g^\bullet$ of the sheaf $\J_X$ of $X$, and let $\varOmega$ be a global nonvanishing $(\dim
\P^N, 0)$-form with values in $\Ok(N+1)$. 
It was shown in \cite[Proposition~3.3]{AS} that there is a (unique) almost
semi-meromorphic current $\omega=\omega_0+\cdots +\omega_{n-1}$ on
$X$, that is smooth on $X_{\text{reg}}$ and such that 
\begin{equation*}\label{poseidon} 
i_*\omega= R\w\varOmega.
\end{equation*}
We say that $\omega$ is a {\it structure form} on $X$. Let
$E^\ell$ denote the restriction of $E_{N-n+\ell}$ to $X$. Then the
component $\omega_\ell$ is an $(n,\ell)$-form taking values in
$\Hom(E^0,E^\ell)$. Moreover, let $X^0=X_\text{sing}$ and
$X^\ell=X_{N-n+\ell}$, where $X_j$ are the BEF-varieties of
$\J_X$. In the language of the previous section $X^\ell$ is the $\ell$th
BEF-variety of the zero sheaf. It follows from that section that 
the $X^\ell$ are independent of the embedding $i:X\to Y$ of $X$
into a smooth manifold $Y$; we therefore call them the
\emph{intrinsic BEF-varieties of $X$}. In light of \eqref{plommon} there are
almost semi-meromorphic forms $\alpha^\ell$, smooth outside $X^\ell$, such that 
\begin{equation}\label{lingon}
\omega_\ell=\alpha^\ell \omega_{\ell-1}.
\end{equation}
on $X$.

\section{Gap sheaves and primary decomposition of sheaves}\label{ass}

Recall that any ideal $\a$ in a Noetherian ring $A$ admits a
\emph{primary decomposition} (or \emph{Noether-Lasker decomposition}), i.e., it can be written as $\a=\bigcap
\a_k$, where $\a_k$ is $\p_k$-primary ($ab\in \a_k$ implies $a\in \a_k$ or
$b^s\in \a_k$ for some $s$ and $\sqrt \a_k=\p_k$) for some prime ideal $\p_k$. The
 primes in a minimal such decomposition are called the \emph{associated
  primes} of $\a$ and the set $\ass (\a)$ of associated primes is independent of
 the primary decomposition.  

\smallskip 

Given a coherent subsheaf $\J$ of $\Ok^X$, Siu \cite{S} gave a way of defining
a ``global'' primary decomposition. 
Let us briefly recall his construction. 
First, for $p=0
,1,\ldots, \dim X$, let $\J_{[p]}\supset \J$ be the $p$th
\emph{gap sheaf (L\"uckergarbe)}, introduced by Thimm
\cite{T}:  A germ $s\in \Ok_x$ is in  $(\J_{[p]})_x$ if and only if there is
a neighborhood $U$ of $x$ and a section $t\in \J(U)$ such that
$s_x=t_x$ and $t_y\in \J_y$ for all $y\in U$ outside an analytic set of
dimension at most $p$. 
It is not hard to see that $\J_{[p]}$ is a coherent sheaf, see
\cite{T}, and that  the set $Y^p$ where
$(\J_{[p]})_x\neq \J_x$ is an analytic variety of dimension at most
$p$, see  \cite[Theorem~3]{S}. 
The irreducible components of $Y^p$, $p=0,1,\ldots, \dim X$, are
called the \emph{associated
  (sub)varieties} of $\J$. 
A coherent sheaf $\J$ is said to be \emph{primary}
if it has only one associated variety $Y$; we then say that $\J$
is \emph{$Y$-primary}. 
Theorem ~6 in \cite{S} asserts 
that each coherent $\J\subset \Ok^X$ admits a decomposition 
\begin{equation}\label{noetherlasker}
\J =\bigcap \J_i, 
\end{equation} 
where there is one $Y_i$-primary intersectand 
$\J_i$ for each associated variety $Y_i$ of $\J$. 
For a radical sheaf $\J_X$, the decomposition
\eqref{noetherlasker} corresponds to decomposing $X$ into irreducible
components.

By Theorem ~4 in \cite{S} if $Y$ is an associated prime variety of
$\J$, then at $x\in X$ the irreducible components $\ass (\J_{Y_x})$ of $Y_x$ are
germs of varieties of associated primes of $\J_x$. Furthermore, if
$Y_x$ is (the variety of) an associated prime of $\J_x$, then $Y_x$ is
contained in $Y^p_x$ for $p\geq \dim Y_x$. For fixed $x$ we get that 
\begin{equation*}%\label{ikea}
\bigcup_{Y\in \ass(\J), Y\ni x}\ass (\J_{Y_x})
\end{equation*}
is a disjoint union of $\ass (\J_x)$. Thus we have
\begin{lma}\label{sauna}
The germ at $x$ of $\J_{[p]}$ is precisely the intersection of the
primary components of $\J_x$ that are of dimension $>p$. 
\end{lma}

Given a subvariety $Z$ of $X$, the \emph{gap sheaf} $\J[Z]\supset\J$ is defined
as follows: A germ $s\in \Ok_x$ is in $\J[Z]_x$ if and only if it
extends to a section of $\J(U)$ for some neighborhood $U$ of $x$, where
$s_y\in\J_y$ for all $y\in U\setminus Z$. Note that 
$\J[Z]_x$ is the
intersection of all components in a primary decomposition of $\J_x$ for which the associated varieties are not
contained in $Z$.  It is not hard to see that $\J[Z]$ is 
coherent, see \cite{T}.  Observe that 
$\J_{[p]}=\J[Y^p]$.

\begin{remark}\label{gronsvart} 
We claim that in fact 
\begin{equation}\label{jamforgap}
\J_{[p]}=\J[Z^\BEF_{n-p}]. 
\end{equation} 
To see this assume first that $X$ is
smooth. 
Then the (germs of)
varieties of associated prime ideals of $\J$ of dimension $\leq p$
are precisely the (germs of)
varieties of associated prime ideals that are contained in
$Z_{n-p}^\BEF$, see, e.g.,
\cite[Corollary~20.14]{Eis}. Now \eqref{jamforgap} follows from Lemma
~\ref{sauna}.

For a general $X$, let $i: X\to Y$ be a local embedding of $X$ into a
manifold $Y$ of dimension $N$ and let $\widetilde \J=\J+\J_X$,
cf. Section \ref{singbef}. It is not hard to verify that if $\a$ is an
ideal in $\Ok^X_x$ and $\tilde  \a:= \a+(\J_X)_x$ is the corresponding ideal in $\Ok^{Y}_x$
then $\a=\cap \a_k$ is a primary decomposition of $\a$ if and only if
$\tilde\a=\cap \tilde\a_k$ is a primary decomposition of
$\tilde\a$. Hence, in light of Lemma ~\ref{sauna},   
$i^*\widetilde \J[V]=\J[V\cap X]$ and $i^* \widetilde \J_{[p]}= \J_{[p]}$. 
By the definition of
BEF-varieties in Section ~\ref{singbef}, thus 
$i^*\widetilde \J[Z_{N-p}^\BEF(\widetilde \J)] = 
\J[Z_{N-p}^\BEF(\widetilde \J)]  = 
\J[Z_{n-p}^\BEF(\J)]$, which proves \eqref{jamforgap} since 
$\widetilde \J_{[p]}=\widetilde \J[Z_{N-p}^\BEF(\widetilde \J)]$ .

\end{remark}

\smallskip

Given a residue current $R$ constructed from a Hermitian locally free resolution
of $\Ok^X/\J$ on a smooth $X$ as in Section ~\ref{rescurr}, in \cite{AW2} we showed
that the germ $R_x$ of the current $R$ at $x\in X$ can be written as $R_x=\sum
R^{\p}$, where the sum is over the associated primes of $\J_x$, and
$R^\p$ has support on the variety $V(\p)$ of $\p$ and has the SEP with
respect to $V(\p)$.

\section{Resolutions of homogeneous ideals}\label{lotta}
Let $\J$ be a coherent ideal sheaf on $\P^N$. Then there is a 
locally free resolution
$\Ok(E_\bullet^f),f^\bullet$, where $E_k$ is a direct sum of line
bundles $E_k= \bigoplus_i \Ok(-d^i_k)$ and $f^k=(f^k_{ij})$
are matrices of homogeneous forms with
$\deg f^k_{ij}= d_k^j-d_{k-1}^i$, see, e.g., \cite[Ch.1,
  Example~1.2.21]{Laz}. 
Let $J$ denote the homogeneous ideal in the graded ring
$\S=\C[z_0,\ldots,z_N]$ associated with  $\J$, and let
$\S(\ell)$ denote the module $\S$ where all degrees are shifted by $\ell$.
Then $\Ok(E^f_\bullet),f^\bullet$ corresponds to a free resolution
\begin{equation}\label{svenne}
\ldots \to\oplus_i \S(-d_k^i)\to\ldots  \to\oplus_i \S(-d_2^i)\to\oplus_i \S(-d_1^i)\to \S
\end{equation}
of the module $\S/J$. Conversely, any such free resolution corresponds
to a locally free resolution
$\Ok(E_\bullet),f^\bullet$.

Recall that the  {\it regularity} of a homogeneous module with a
minimal graded free resolution
\eqref{svenne} is defined as $\max_{k,i}(d_k^i-k)$,
see, e.g., \cite[Ch.4]{Eis2}. 
The \emph{regularity} $\reg J$ of the ideal $J$ equals $\reg
(\mathcal S/J)+1$, cf.  \cite[Exercise~4.3]{Eis2}.

If $X$ is a subvariety of $\P^N$, then the
\emph{regularity} of $X$,  $\reg X$, is defined as the regularity of $J_X$. 
Notice that if $X$ has pure dimension, then the ideal $J_X$ has pure dimension in $\S$; in particular
the ideal associated to the origin is not an associated prime ideal. 
Theorem~20.14 in \cite{Eis} thus implies that 
$Z_0^{\BEF}$ is empty. Therefore 
the depth of $\S/J_X$ is at least $1$, and hence a minimal
free resolution of $\S/J_X$ has length $\le N$.
For such a resolution we thus get 
\begin{equation}\label{spasm}
\reg X = \max_{k\le \min (M,N)} (d_k^i-k) +1.
\end{equation}
A global section of $\Ok(s)|_X\to X$ extends to a global section of
$\Ok(s)\to \P^N$ as soon as $s\geq \reg X-1$, see, e.g.,
\cite[Chapter~4]{Eis2}.

\section{Division problems on singular varieties}\label{divprob}

Let 
 $E_\bullet^g, g^\bullet$ be a complex that corresponds to a Hermitian free resolution of $\Ok^{\P^N}/\J_X$ as above, and let  $E^f_\bullet,f^\bullet$ be an
arbitrary Hermitian pointwise generically surjective complex over $\P^N$. Then the product current  
\begin{equation*}
R^f\w R^g:=R^{f,\lambda}\w R^g|_{\lambda=0}
\end{equation*}
is well-defined on $\P^n$, 
\begin{equation*}
R^f\w\omega := R^{i^*f,\lambda}\w\omega |_{\lambda=0}
\end{equation*} 
is a well-defined current on $X$, and $i_*(R^f\w\omega)= R^f\w R^g$, see
\cite[Section~2]{semester}. In particular, $R^f\w R^g$ and $R^f\w
\omega$ only depend on the restriction of $f$ to $X$, and thus these
currents are well-defined even if $f$ is only defined over $X$. 
Moreover $R^f\w R^g \phi=0$ if and only if $R^f\w \omega
i^*\phi=0$. 
On $X_{\text{reg}}$, $R^f\w\omega$ is just the product of the current $R^f$
and the smooth form $\omega$.

The current $R^f\w R^g$ is related to the  tensor product complex $E^h_\bullet, h^\bullet$, where
\begin{equation*}%
E_k^h=\bigoplus_{i+j=k} E^f_i\otimes E^g_j,
\end{equation*}
and $h=f+g$, 
cf. \cite[Section~2.5]{semester}, in a similar way as is the current
$R^h$ associated with this complex,  see \cite{A12}. 
In particular, if $\phi$ is a section of $E^h_0=E^f_0\otimes E^g_0$
such that  $R^f\w R^g\phi=0$, 
one can locally solve $f^1q+g^1q'=\phi$. 
Moreover if
\eqref{cohovillkor} is satisfied for the product complex there is a
global such section $(q,q')$ of $E^h_1=E^f_1\otimes E^g_0 \oplus
E^f_0\otimes E^g_1$.  In general,
however, $R^f\w R^g$ does not coincide with $R^h$.

In fact, the definition of $R^f$ in Section ~\ref{rescurr} works also
when $Y$ is singular. However, Proposition ~\ref{glatta} and the
duality principle do not hold in general, see, e.g., \cite{Lar2}, and
therefore $R^f$ itself is not so well suited for division problems.

\begin{ex}\label{koszuligen}

Assume that $E^f_\bullet,f^\bullet$ is the Koszul complex
generated by sections $f_j$ of $L_j=\Ok(d_j)|_X$, where $X\subset \P^N$, twisted by $S=\Ok(\rho)$,
as in Example ~\ref{koszulex}, and that $E^g_\bullet,g^\bullet$ is a complex
associated with a minimal Hermitian free resolution of $\mathcal S/J_X$ as in Section~\ref{lotta}. 
Note that then $E^h_\ell$ is a direct sum of line bundles
\[
\Ok(\rho-(d_{i_1}+\cdots + d_{i_\ell}) -d^i_{k-\ell}).
\] 
Recall that 
\begin{equation}\label{pnko}
H^k(\P^N,\Ok(\ell))=0 \quad\quad  \text{if}\quad  \ell\ge -N \quad
\text{or}\quad k<N, 
\end{equation} 
see,  e.g., \cite{Dem}. 
Thus \eqref{cohovillkor} is satisfied
if $\rho\geq d_{i_1}+\cdots + d_{i_\ell} + d^i_{N+1-\ell}-N$ for
$\ell=1,2,\ldots, \min (m,n+1)$ and all choices of $i$ and $i_j$. 
Notice that, cf., \eqref{spasm}, 
$$
 d_{N+1-\ell}^i-N =\big(d_{N+1-\ell}^i-(N+1-\ell)\big)+1-\ell\leq \reg
 X - \ell. 
$$
Hence \eqref{cohovillkor} is satisfied if  
\begin{equation}\label{enda2}
\rho\ge d_1+\cdots + d_{min(m,n+1)} - \min(m,n+1)+\reg X.
\end{equation}
Summing up we have: 
\smallskip 

\noindent
{\it If $\rho$ satisfies \eqref{enda2}
and $\phi$ is a section of $\Ok(\rho)$ on $\P^N$ such that
$R^f\w R^g\phi=0$ (or equivalently $R^f\w R^gi^*\phi=0$) then there are global sections $q_j$ of
$\Ok(\rho-d_j)$ such that $f_1q_1+\cdots +f_mq_m=\phi$ on $X$.}

\smallskip 
\noindent 
If $X$ is Cohen-Macaulay we may assume that $E^g_\bullet, g^\bullet$
ends at level $N-n$. If moreover $m\leq n$, then $E^h_\bullet,
h^\bullet$ ends at level $\leq N$ and thus \eqref{cohovillkor} is
satisfied for any $\rho$. 

\end{ex}

\begin{ex}\label{exact}
Let $F_j$ be polynomials in $\C^N$, let $\hat f_j$ be the
sections of $\Ok(\deg F_j)\to \P^N$ corresponding to $F_j$, and let $\J_{\hat f}$ be the ideal
sheaf on $\P^N$ generated by the $\hat f_j$. 
Moreover, let 
$E^f_\bullet, f^\bullet$ and $E^g_\bullet, g^\bullet$ be complexes
associated with minimal free resolutions of $\J_{\hat f}$ and $\J_X$ as in
Section ~\ref{lotta}, where $X$ is a subvariety of $\P^N$; say $E_k^f=\bigoplus \Ok(\delta_k^i)$ and   
$E_k^g=\bigoplus \Ok(d_k^i)$. 
Then $E^h_k$ is a direct sum of line bundles
$\Ok(-\delta_\ell^i-d^j_{k-\ell})$, and thus \eqref{cohovillkor} is satisfied
if $\rho\geq \delta_\ell^i + d^j_{N+1-\ell}-N$ for all $i,j, \ell$,
cf.\ Example ~\ref{koszuligen}. 
Notice that, in light of Section ~\ref{lotta}, 
\begin{equation*}
\delta_\ell^i + d^j_{N+1-\ell}-N= 
(\delta_\ell^i-\ell) + (d^j_{N+1-\ell}-(N+1-\ell))+1
\leq \reg J_{\hat f}+\reg X-1, 
\end{equation*} 
where $J_{\hat f}$ is the homogeneous ideal associated with $\J_{\hat
  f}$. Thus \eqref{cohovillkor} is satisfied if $\rho\ge \reg J_{\hat f}+\reg X-1$.

Let $Z^{\hat f}_k$ and $Z^g_\ell$ be the BEF-varieties of
$\J_{\hat f}$ and $\J_X$,
respectively. 
Theorem ~4.2 in \cite{A12} asserts that if 
\begin{equation}\label{skara2}
\codim (Z^{\hat f}_k\cap Z^g_\ell)\geq k+\ell, 
\end{equation}
then $R^f\w R^g\phi =0$ if and only if $\phi\in \J_{\hat
  f}+\J_X=\J_{f}+\J_X$, where $\J_f$ is the sheaf on $X$ generated by
the restrictions $f_j$ of $\hat f_j$, cf. Section ~\ref{singbef}. 
If moreover $\J_{\hat f}$ and $\J_X$ are both Cohen-Macaulay and the
resolutions $\Ok(E^{f}_\bullet), f^\bullet$ and $\Ok(E^g_\bullet),
g^\bullet$ have minimal length, then $R^f\w R^g = R^h$, see \cite[Theorem~4.2]{A12}. 
\end{ex}

\subsection{Distinguished varieties}\label{dist}

Let $X$ be a subvariety of $\P^N$ and let $\tilde f_j$ be sections of $L=\Ok(d)|_X$. Moreover, let
$
\blow\colon X_+\to X
$
be the normalization of the blow-up
of $X$ along $\J_{\tilde f}$, and let $W=\sum r_j W_j$ be the exceptional divisor; here $W_j$ are
irreducible Cartier divisors. The images $Z_j:=\blow(W_j)$ are called
the {\it (Fulton-MacPherson) distinguished varieties}   associated
with $\J_{\tilde f}$, see, e.g., \cite{Laz}.
If we consider ${\tilde f}=({\tilde f}_1,\ldots,{\tilde f}_m)$ as a section of $E^*:=\oplus_1^m \Ok(-d)$, then
 $\blow^* {\tilde f}={\tilde f}^0{\tilde f}'$, where ${\tilde f}^0$ is a section of the line bundle
$\Ok(-W)$  and ${\tilde f}'=({\tilde f}_1',\ldots,{\tilde f}_m')$ is a nonvanishing section of
$\blow^* E^*\otimes \Ok(W)$, where  $\Ok(W)=\Ok(-W)^{-1}$.
Furthermore,
$
\omega_{\tilde f}:=dd^c\log|{\tilde f}'|^2
$
is a smooth first Chern form for $\blow^* L\otimes\Ok(W)$. 
We will use the geometric estimate
\begin{equation}\label{elest}
\sum r_j \deg_L Z_j\le \deg_L X
\end{equation}
from  \cite[Proposition~3.1]{EL},
see also \cite[(5.20)]{Laz}.

\smallskip 

Let $R^{\tilde f}$ be the residue current associated with the Koszul complex of
the ${\tilde f}_j$ as in Example ~\ref{koszulex} and consider the regularization \eqref{mung} of $R^{\tilde f}$. 
Using the notation in Example~\ref{koszulex}, 
$
\blow^*\sigma=(1/{\tilde f}^0)\sigma',
$
where $1/{\tilde f}^0$ is a meromorphic section of $\Ok(W)$ and $\sigma'$ is a smooth
section of $\blow^*E\otimes\Ok(-W)$. It follows that
\begin{equation*}
\blow^*(\sigma\w(\dbar\sigma)^{k-1})=\frac{1}{({\tilde f}^0)^k}\sigma'\w(\dbar\sigma')^{k-1},
\end{equation*}
and hence
\begin{equation*}
\blow^*R^{\lambda}_k=\dbar |{\tilde f}^0{\tilde f}'|^{2\lambda}\w
\frac{1}{({\tilde f}^0)^k}\sigma'\w(\dbar\sigma')^{k-1} \text{ for } \Re\lambda
>> 0,
\end{equation*} 
when $k\ge 1$. 
Since ${\tilde f}'$ is nonvanishing, by \eqref{nkel} the value at $\lambda=0$
is precisely 
\begin{equation}\label{stuga1}
R_k^+:=\dbar\frac{1}{({\tilde f}^0)^k}\w \sigma'\w(\dbar\sigma')^{k-1}.
\end{equation}
Thus 
\begin{equation*}%\label{stuga}
\blow_* R^+_k=R^{\tilde f}_k.
\end{equation*}

\section{Proofs}\label{proofs}

\begin{proof}[Proof of Theorem~\ref{noethersats}]

For $j=1,\ldots, m$, let $\hat f_{j}$ be the $\deg F_j$-homogenization of
the polynomial $F_j$, considered as a section of $\Ok(\deg F_j)\to
\P^N$.   
Moreover let $g_1,\ldots, g_r$ be global generators of the ideal
sheaf $\J_X$; assume they are sections of $\Ok(d_1),\ldots, \Ok(d_r)$,
respectively. Let $\J=\J_{\hat f} +\J_X=\J_f+\J_X$. 
Then there is a locally free resolution $\Ok(E_\bullet^h), h^\bullet$
of $\Ok/\J$, where each $E_k^h$ is a direct sum of line bundles $E_k=\bigoplus_i \Ok(-d^i_k)$
and in particular $E^1=\bigoplus_1^m \Ok(-\deg F_j)\oplus_1^r
\bigoplus \Ok (-d_k)$
and $h^1=(f_1,\ldots, f_m,g_1,\ldots, g_r)=:f+g$, cf.\ Section
~\ref{lotta}. 
Let $R=R^h$ be the residue current associated with $E_\bullet
^h,h^\bullet$.

 Recall from Section
~\ref{ass} that for fixed $x\in X$, $R_x=\sum R^\p$, where the sum is over $\ass (\J_x)$  
and where $R^\p$ has the SEP with respect to $V(\p)$; in particular,
$\mathbf 1_{H_\infty} R^\p=R^\p$ if $V(\p)\subset H_\infty$ and
$\mathbf 1_{H_\infty} R^\p=0$ otherwise. 
Thus   
\begin{equation}\label{artist}
\mathbf 1_{H_\infty} R_x=\sum_{\p\in \ass (\J_x), V(\p) \subset
  H_\infty}R^\p. 
\end{equation}
In Remark \ref{gronsvart} we saw that $\a=\cap \a_k$ is a primary decomposition of
the ideal $\a$ in $\Ok^X_x$ if and only if $\tilde\a=\cap \tilde\a_k$ is
a primary decomposition of the ideal 
$\tilde\a= \a+(\J_X)_x$ in $\Ok^Y_x$. 
Thus, that $\J_f$ has no associated varieties contained in $X_\infty$
implies that,  
for a fixed $x\in X$, $\J_x$ has no (varieties of)
associated primes contained in the hyperplane $H_\infty$ at infinity 
in $\P^N$. 
We conclude, in light of \eqref{artist}, that $\mathbf
1_{H_\infty}R =0$. 
If $\phi$ is any homogenization of $\Phi$ then  $\mathbf
1_{\C^N}R\phi=0$ because of the
duality principle and hence
$R\phi=\mathbf 1_{H_\infty} R\phi +\mathbf 1_{\C^N}R\phi =0$.

Assume that the complex $E^h_\bullet,h^\bullet$ ends at level $M$ (by
Hilbert's syzygy theorem we may assume that $M\leq N+1$) and let
\begin{equation}\label{betadef}
\beta:=\max_i d_{N+1}^i-N \text{ if } M=N+1 \quad \text{ and } \beta:=0
\text{ otherwise.}
\end{equation} 
If $\rho\geq \beta$ then  \eqref{cohovillkor} is satisfied for
$E^h_\bullet, h^\bullet$ twisted by $\Ok(\rho)$ in light of
\eqref{pnko} and thus
by Proposition ~\ref{glatta} there are global
holomorphic sections $q= (q_j)$ of $\bigoplus \Ok(\rho-\deg F_j)$
and $q'=( q'_k)$ of $\bigoplus \Ok(\rho-d_k)$
over $\P^N$ such that 
$\hat f q + g q'=\phi$. Indeed, recall from the end of Section ~\ref{rescurr} that $R$ is
also the residue current associated with the twisted complex. 
Dehomogenizing gives polynomials $Q_j$, $Q_j'$, and $G_j$ in $\C^N$ such that 
\begin{equation*}%\label{lava}
\sum F_j Q_j + \sum G_j Q_j' =\Phi
\end{equation*} 
and where $\deg (F_j Q_j)\leq \rho$. Since the $G_j$ vanish on $V$ we
get the desired solution to \eqref{hummer} on $V$, and thus 
the first part of Theorem ~\ref{noethersats} follows with $\beta$ as in
\eqref{betadef}. 

\smallskip 

If $V=\C^N$, $\Ok_X$ should be interpreted as the zero sheaf. Then
$E^h_\bullet, h^\bullet$ is a locally free resolution of $\Ok/\J_f$
and $\beta\leq \reg J_f$,  cf.\ Section
~\ref{lotta}.

\medskip

For the second part of Theorem ~\ref{noethersats}, assume that $\J_f$
has an associated variety contained in $X_\infty$. 
We are to prove that for arbitrarily large $\ell$ there is a 
polynomial $\Phi=\Phi_\ell$ of degree $\geq \ell$ in $(F_j)$ on
$V$ for which one can not solve \eqref{hummer} with $\deg (F_jQ_j)\leq \deg\Phi_\ell$.

Let $L=\Ok(1)|_X$. The hypothesis
on $\J_f$ then means that $\J_f[X_\infty]$ is strictly larger than
$\J_f$.  Therefore, since $L$ is ample, for some large
enough $s_0$ there is a global section $\psi_0$ of
$L^{\otimes s_0}\to X$ such that $\psi_0$ is in $\J_f[X_\infty]$ but not in
$\J_f$. 
Moreover we can find a global section $\psi$ of $L^{\otimes s}$ for some $s\geq 1$
such that $\psi$ does not vanish identically on any of the
associated varieties of $\J_f$ that are contained in $X_\infty$. 
We may assume that $s_0,s\geq \reg X-1$, so that $\psi_0$ and $\psi$
extend to global sections $\hat \psi_0$ and $\hat \psi$ of $\Ok(s_0)$
and $\Ok(s)$, respectively. Let $\Psi_0$ and $\Psi$ be the
corresponding dehomogenized polynomials in $\C^N$. 
For $\ell \geq 0$, let $\phi_\ell=
\psi_0\psi^\ell$ and $\Phi_\ell=\Psi_0\Psi^\ell$. 
Since $\J_f[X_\infty]_x=(\J_f)_x$ for all $x\in V$, 
$\Phi_\ell$ is in the ideal $(F_j)$ on $V$, and thus we can
solve \eqref{hummer} for $\Phi=\Phi_\ell$ on $V$. 
Assume that there is a solution to \eqref{hummer} with $\deg
(F_jQ_j)\leq \rho_\ell$. 
Then there are sections $q_j$ of $L^{\rho_\ell-\deg F_j}$ such that 
\begin{equation*}%\label{forsok}
\sum f_j q_j = z_0^{\rho_\ell-(s_0+s\ell)}\phi_\ell
\end{equation*}
on $X$. Since $\phi_\ell$ is not in $\J_f$ it follows that 
$\rho_\ell-(s_0+s\ell)\geq 1$ and 
thus $\rho_\ell\geq 1+(s_0+s\ell)\geq 1 +\deg\Phi_\ell$. 
Since $\hat\psi$ does not vanish identically at $X_\infty$, $\deg \Psi
\geq 1$ and hence $\deg\Phi_\ell\geq\ell$. 
Hence we have found $\Phi_\ell$ with the desired properties and the
second part of Theorem ~\ref{noethersats} follows.

\end{proof}

\begin{remark}\label{specfall}
If $\J_{\hat f}$ and $\J_X$ are Cohen-Macaulay and the BEF-varieties of $\J_{\hat f}$ and $\J_X$ satisfy \eqref{skara2}, then
we can choose the complex $E^h_\bullet, h^\bullet$ in the above proof
to be the tensor product of the complexes $E^f_\bullet, f^\bullet$ and
$E^g_\bullet, g^\bullet$ corresponding to minimal resolutions of
$\J_{\hat f}$ and $\J_X$, see Example ~\ref{exact}. 
In this case, by Example ~\ref{exact}, we get Theorem
~\ref{noethersats} for $\beta=\reg J_{\hat f}+ \reg X  -1$. 

\end{remark} 

The residue current technique in the preceding proof is convenient
and makes it possible to carry out the proof within our
general framework, 
but it is not
crucial.

\begin{remark}[The algebraic approach]\label{lock}

Let us first sketch an algebraic proof of the first part of
Theorem~\ref{noethersats}.  We use the notation from the proof above. 
To begin with we have to prove that $\phi$ is in $\J$, which of course
precisely corresponds to proving that $R\phi =0$. 
Since (the restriction to $V$ of) $\phi$ is in 
$\J_f$ on $V$ it follows that $\phi_{x'}$ is in $\J$ outside $H_\infty$. Since
moreover $\J=\Ok^{\P^N}$  outside $X$, we have to prove that $\phi_x\in\J_x$
for each $x\in X_\infty$.  At such a point $x$ we have a minimal primary
decomposition $\J_x=\cap_\ell \J^\ell_x$. 
Since $\J$ is coherent, $\J\subset\J^\ell$ in a neighborhood $\U$ of
$x$, where $\J^\ell$ is the coherent sheaf defined by $\J^\ell_x$. 
Let $Z^\ell$ be the zero-set of
$\J^\ell$.  Since $\phi_{x'}$ is in
$\J_{x'}$ for $x'$ outside $H_{\infty}$ it follows that $\phi_{x'}$ is in
$\J^\ell_{x'}$ for $x'\in Z^\ell\setminus H_\infty$. 
Hence $\F:=(\J^\ell+(\phi))/\J^\ell$ is
a coherent sheaf in $\U$ with support on $Z^\ell\cap H_\infty$.
Since by assumption $\J_f$ has no associated varieties contained in $X_\infty$ it
follows that $Z^\ell\cap H_\infty$ has positive codimension in
$Z^\ell$, cf.\ the proof of Theorem~\ref{noethersats} above. 
Therefore, by the
Nullstellensatz there is a holomorphic function $h$, not vanishing
identically on $Z^\ell$ such that $h\F=0$. 
In particular, $h_x\phi_x\in \J^\ell_x$. Since $h_x$ is not in the radical
of $\J^\ell_x$ and $\J^\ell_x$ is primary it follows that
$\phi_x\in\J^\ell_x$. We conclude that $\phi_x\in\J_x$. 
Notice that the last arguments above can be thought of as an algebraic version of the
SEP-argument in the proof of Theorem~\ref{noethersats} above.

Next we would like to use that $\phi\in\J$ to conclude
that there is a global holomorphic solution to $hq=\phi$. 
By a
partition of unity, using that $E^h_\bullet, h^\bullet$ is exact, one
can glue local such solutions together to obtain a global smooth solution to
$(h-\dbar)\psi=\phi$, cf. \cite[Section~4]{semester}. By solving a certain sequence of
$\dbar$-equations in $\P^N$ we can modify $\psi$ to a global
holomorphic solution $q$ to $hq=\phi$. These $\dbar$-equations are solvable if
$\rho\geq \beta$ defined by \eqref{betadef}. Alternatively, one can
directly refer to the well-known
result that there is a solution to $hq=\phi$
if $\rho\ge\reg J$, where $J$ is the homogeneous ideal corresponding
to $\J$, see, e.g., \cite[Proposition~4.16]{Eis2}. 

\smallskip
In the same way Theorem ~\ref{macsats} and ~\ref{maxsats} follow  without any reference to residues.
Probably one can also find give  an algebraic proof of Theorem \ref{max2}.

\end{remark}
In the next proof the residue
technique plays a more decisive role.

\begin{proof}[Proof of Theorem~\ref{halvglattnoether}]

Let  
\[
\rho=\max(\deg\Phi+\mu d^{c_\infty}\deg X, (d-1)\min(m,n+1)+\reg X),
\]
or if $X$ is Cohen-Macaulay and $m\leq n$ let 
$\rho=\deg\Phi+m d^{c_\infty}\deg X$, 
and let $\phi$ be the $\rho$-homogenization of $\Phi$ considered as a section of $\Ok (\rho)|_X$. Note that then $\phi=z_0^{\rho-\deg
  \Phi} \tilde \phi$, where $\tilde\phi$ is the $\deg
\Phi$-homogenization of $\Phi$. Moreover, let 
$R^{\tilde f}\wedge\omega$ be the residue current associated with the
(twisted) Koszul complex $E_\bullet^{\tilde f},{\tilde f}^\bullet$ of
the sections $\tilde f_j$ of $\Ok(d)|_X$ associated with $F_j$, and a complex $E^g_\bullet, g^\bullet$ associated with
a minimal resolution of $\Ok/\J_X$ as in Example ~\ref{koszuligen} (with $d_j=d$ for all $j$).

\smallskip 

\noindent 
\textbf{Claim:} $R^{\tilde f}\w\omega_0\phi$ has support on $Z^{\tilde f}\cap X^0$.

\smallskip 

To prove the claim, since $\omega$ is smooth on $X_{\text{reg}}$, it is enough to
show that $R^{\tilde f}\phi=0$ on $X_{\text{reg}}$. First, since $\codim Z^{\tilde f}\cap V\geq m$, the duality principle for a complete
intersection, cf.\ Example ~\ref{koszulex}, implies that  
$ R^{\tilde f}\phi=0$ on $V_{\text{reg}}$.

Next, to prove that $\mathbf 1_{X_\infty\setminus X^0}R^{\tilde f}\phi=0$ we consider the normalization of the blow-up $\blow\colon X_+ \to X$,
and let $R^+:=\sum R_k^+$ be as in Section ~\ref{dist}. 
Let $W'$ be the union of the irreducible components of $W=\blow^{-1}Z^{\tilde f}$ that are contained in
$\blow^{-1}X_\infty$. We claim that 
\begin{equation}\label{sport2}
\1_{X_\infty} R^{\tilde f}=\blow_*\big(\1_{W'}R^+\big). 
\end{equation}
In fact, by \eqref{enkelstjarna},  
\begin{equation}\label{busar}
\1_{X_\infty} R^{\tilde f}=\blow_*\big(\1_{\blow^{-1}X_{\infty}}R^+\big)=
\blow_*(\1_{\blow^{-1}X_\infty}(\1_{W'}+\1_{W\setminus W'}) R^+\big). 
\end{equation}
By, \eqref{dubbelstjarna}, 
$
\1_{\blow^{-1}X_\infty}\1_{W'}R^+=\1_{W'}R^+$. 
Moreover, 
$$
\1_{\blow^{-1}X_\infty}\1_{W\setminus W'}\dbar\frac{1}{({\tilde f}^0)^k}= 
\1_{\blow^{-1}X_\infty\cap (W\setminus W')}\dbar\frac{1}{({\tilde f}^0)^k}
=0
$$
by \eqref{dubbelstjarna} and the dimension principle, since $\blow^{-1}X_\infty\cap (W\setminus W')$ has codimension
at least $2$ in $X_+$. In view of  \eqref{stuga1}   we conclude that
 $\1_{\blow^{-1}X_\infty}\1_{W\setminus W'} R^+=0$, and thus \eqref{sport2}  follows from \eqref{busar}.

It follows from \eqref{sport2} that $\mathbf 1_{X_\infty\setminus
  X^0}R^{\tilde f} \phi=0$ if 
$\1_{W'}R^+\blow^*\phi=0$. To show that $\1_{W'}R^+\blow^*\phi$ vanishes first note that it is sufficient
to show that it vanishes in a \nbh of each point
$x$ on $W'$ where $W$ is smooth. Indeed, since $W_{\text{sing}}$ has
codimension at least $2$ in $W$, 
$\mathbf 1_{W_{\text{sing}}} \dbar(1/({\tilde f}^0)^k)=0$ by the
dimension principle. Hence, using \eqref{stuga1} and
\eqref{dubbelstjarna} we get that 
\[
\mathbf 1_{W'} R^+ = 
\mathbf 1_{W'} (\mathbf 1_{W_{\text{reg}}} + \mathbf
1_{W_{\text{sing}}}) R^+ =\mathbf 1_{W'\cap W_{\text{reg}}} R^+.
\]
Consider now $x\in 1_{W'\cap W_{\text{reg}}}$; say $x$ is contained in
the irreducible component
$W_j$ of $W'$. 
In a \nbh of $x$ we have that ${\tilde f}^0=s^{r_j} v$,
where $s$ is a local coordinate function and $v$ is nonvanishing and $r_j$ is as in Section~\ref{dist}.
Since $\phi = z_0^{\rho-\deg\Phi} \tilde \phi$, by the choice of $\rho$, $\blow^*\phi$ vanishes to order (at least)
$\mu d^{c_\infty}\deg X$ on $W'$. 

If $\Omega$ is a first Chern form for $\Ok(1)|_X$, e.g., $\Omega=dd^c\log|z|^2$, then $d\Omega$ is
a first Chern form for $L=\Ok(d)|_X$ on $X$ (notice that $d$ denotes the degree and not the differential).
By \eqref{elest} we therefore have that
$$
r_j\int_{Z_j}(d \Omega)^{\dim Z_j}\le \int_X (d\Omega)^n, 
$$
which implies that
\begin{equation*}%\label{sunnanvind}
r_j\le d^{\codim Z_j}\deg X.
\end{equation*}
It follows that 
$\blow^*\phi$ vanishes (at least) to order  $\mu r_j$
on $W_j$ and hence it has a factor $s^{\mu r_j}$.
In a \nbh of $x$,   
$$
\dbar\frac{1}{({\tilde f}^0)^k}=\dbar \frac{1}{s^{k r_j}}\w smooth 
$$
and thus, in light of \eqref{stuga1}, 
$R_k^+\blow^*\phi=0$ for $k\leq
\mu$ there.  
Hence $\mathbf
1_{W'\cap W_{\text{reg}}}R_k^+\blow^*\phi=0$ for $k\leq
\mu$ and  $\mathbf 1_{X_\infty
  \setminus X^0} R^{\tilde f}\phi =0$. 
We conclude that $\mathbf 1_{X\setminus X^0} R^{\tilde f}\phi= 
\mathbf 1_{V_{\text{reg}}} \ R^{\tilde f}\phi + \mathbf 1_{X_\infty
    \setminus X^0} R^{\tilde f}\phi=0$, which proves the claim that $R^{\tilde
  f}\w\omega_0\phi$ has support on $Z^{\tilde f}\cap X^0$.

\smallskip

By \eqref{paxa} and the dimension principle
we conclude that $R^{\tilde f}\w\omega_0\phi$  vanishes identically, since the bidegree
of $R^{\tilde f}$ is at most $(0,m)$ and $\omega_0$ has bidegree
$(n,0)$. 
Thus $ R^{\tilde f}\w\omega_1\phi= R^{\tilde f}\w\alpha^1\omega_0\phi$, see \eqref{lingon}, vanishes
outside $X^1$. By \eqref{paxa} and the dimension principle,
it vanishes identically since the bidegree of $ R^{\tilde
  f}\w\omega_1$ is at most $(n,m+1)$. 
By induction, it follows that
$ R^{\tilde f}\w\omega_\ell\phi=0$ for each $\ell$. We conclude that
$ R^{\tilde f}\wedge\omega\phi=0$. 

\smallskip 

Since $\rho$ satisfies \eqref{enda2} (with $d_j=d)$ and $R^{\tilde
  f}\wedge\omega\phi=0$, by Example ~\ref{koszuligen} there is a
global section $q=(q_j)$ of $\sum_1^m \Ok(\rho-d)$ such that
$fq=\phi$ on $X$. Dehomogenizing gives polynomials $Q_j$ such that
\eqref{hummer} holds on $V$ and $\deg (F_jQ_j)\leq \rho$.

\end{proof}

\begin{proof}[Proof of Theorems ~\ref{macsats} and  ~\ref{max2}]
Let  
\[
\rho=\max(\deg\Phi, d_1+\ldots +d_{\min(m,n+1)} -\min(m,n+1) +\reg
X),
\]
or if $X$ is Cohen-Macaulay and $m\leq n$ let $\rho=\deg \Phi$. 
Moreover let $\phi$ be the $\rho$-homogenization
of $\Phi$ and let 
$R^{f}\wedge\omega$ be the residue current associated with the twisted Koszul
complex $E_\bullet^{f},{f}^\bullet$ of the 
$\deg F_j$-homogenizations $f_j$ of $F_j$ and a minimal
resolution of  $\Ok/\J_X$ as in 
Example ~\ref{koszuligen}.

We claim that under the hypotheses of both theorems $R^f\w\omega_0\phi$ has
support on $Z^f\cap X^0$. Since $\omega$ is smooth outside $X^0$ it is
enough to show that $R^f\phi=0$ there. First in the case of Theorem
~\ref{macsats}, $R^f$ vanishes for trivial reasons, 
since $Z^f$ is empty. In the case of Theorem ~\ref{max2}, first
$R^f\phi$ vanishes on $V_{\text{reg}}$ by the duality 
principle. Next, since by assumption \eqref{krax} holds and $Z^f$ has no
irreducible components in $X_\infty$, it holds that $\codim
(X_\infty \cap Z^f)>m$. Since the components of $R^f$ have
bidegree at most $(0,m)$, we conclude that $\mathbf
1_{X_\infty\setminus X_0} R^f=0$ by the dimension principle. This 
proves that $R^f\w\omega\phi$ has
support on $Z^f\cap X^0$. 

Now arguing as in the end of the proof of Theorem
~\ref{halvglattnoether}, we get that $R^f\w\omega\phi=0$, and the
results follow from Example \ref{koszuligen}.

\end{proof}

\begin{remark}\label{gron}
If $\deg F_j=d$, then Theorems ~\ref{macsats} and ~\ref{max2} follow 
directly from Theorem ~\ref{halvglattnoether}. 
First, notice that Theorem ~\ref{macsats} follows if we apply Theorem
~\ref{halvglattnoether} to $F_j$ with no common zeros on $X$. Indeed,
since $Z^f$ is empty, $\codim (Z^f\cap X)=\infty$ and thus 
\eqref{krax2} and \eqref{paxa} are satisfied, and moreover
$c_\infty=-\infty$.

Next, assume that $F_j$ satisfy the hypothesis of Theorem
~\ref{max2}. Since the codimension of a distinguished variety is at
most $m$ the condition that $Z^f$ satisfies \eqref{krax} and has no
irreducible component contained in $X_\infty$ means that \eqref{krax2}
is satisfied and no distinguished varieties can be contained
in $X_\infty$. Thus $c_\infty=-\infty$ and 
 $d^{c_\infty}=0$ and Theorem
~\ref{max2} follows from Theorem ~\ref{halvglattnoether}. 

\end{remark}

\def\listing#1#2#3{{\sc #1}:\ {\it #2},\ #3.}

\end{document}